\documentclass[11pt,twoside, final]{amsart}
\copyrightinfo{0}{Iranian Mathematical Society}
\usepackage{amsmath,amsthm,amscd,amsfonts,amssymb,enumerate}
\usepackage{graphicx}		
\usepackage{color}
\usepackage{listings}
\usepackage[colorlinks]{hyperref}

\newtheorem{theorem}{Theorem}[section]
\newtheorem{proposition}[theorem]{Proposition}

\newtheorem{corollary}[theorem]{Corollary}
\theoremstyle{definition}

\theoremstyle{remark}
\newtheorem{remark}[theorem]{Remark}
\numberwithin{equation}{section}

\begin{document}


\title[Finite Wavelet Frames over Prime Fields]{A Constructive Approach to the Finite Wavelet Frames over Prime Fields}

\author[ A. Rahimi and N. Seddighi]{Asghar Rahimi and Niloufar Seddighi}
\address{Asghar Rahimi, Faculty Mathematics, University of Maragheh, Maragheh, Iran.}
\email{rahimi@maragheh.ac.ir}
\address{Niloufar Seddighi, Faculty Mathematics, University of Maragheh, Maragheh, Iran.}
\email{stu\_seddighi@maragheh.ac.ir}

%
%
%

%

\begin{abstract}
 In this article, we present a constructive method for computing the frame coefficients of finite wavelet frames
over prime fields using tools from computational harmonic analysis and group theory.\\\\
\textbf{Keywords:}  Finite wavelet frames, finite wavelet group, prime fields.\\
\textbf{MSC(2010):}  Primary 42C15, 42C40, 65T60; Secondary 30E05, 30E10.
\end{abstract}

\maketitle

\section{\bf Introduction}

The mathematical theory of frames in Hilbert spaces was introduced in \cite{duf.sch}, and has been studied
in depth for finite dimensions in \cite{P.Bal, ff}.
Finite frames have been applied as means to interpret periodic signals and
privileged in areas ranging from digital signal processing to image analysis, filter banks, big data and compressed
sensing, see \cite{ff, san.vinc, strang} and references therein. The representations of a  function/signal in
time-frequency (resp. time-scale) domain are obtained through analyzing the signal
with respect to an over complete system whose elements are localized in time and frequency (resp. scale) \cite{arefi.z1, arefi.z2, AGHF1, AGHF.KAM}.

In the framework of wave packet analysis \cite{Chris.Rahimi},
finite wavelet systems are particular classes of finite wave
packet systems which have been introduced recently in \cite{AGHF.wpt.ff, AGHF.wpt.n, AGHF.cwps.p}.
The analytic structure of finite wavelet frames over prime fields (finite Abelian groups of prime orders)
has been studied in \cite{AGHF.SFWFPF, AGHF.MMP}.
The notion of a wavelet transform over a prime field was introduced in
\cite{f.g.h.t} and extended for finite fields in \cite{f.m.s, j}. This notion is
an example of the abstract generalization of wavelet transforms using harmonic analysis,
see \cite{Arefi1, Arefi2, cai, Ten, fu05, san.vinc} and
classical references therein.
The current paper consists of a constructive approach to the abstract aspects
of nature of finite wavelet
systems over prime fields. The motivation of this paper is to
establish an alternative constructive formulation for the wavelet coefficients of
finite wavelet frames over prime fields.

This article is organized as follows.
Section 2 introduces some notations as well as a brief review of Fourier transform on finite
cyclic groups, periodic signal processing, and finite frames.
Then in section 3 we present a constructive technique for computing the frame coefficients of finite wavelet frames
over prime fields using tools from computational harmonic analysis and theoretical group theory.
In addition, we shall also give a constructive characterization for frame conditions of
finite wavelet systems over prime fields using matrix analysis terminology.

\section{\bf Preliminaries}

This section is devoted to present a brief review of notations, basics, and preliminaries of
Fourier analysis and computational harmonic analysis over finite cyclic groups,
for any details we refer the readers to \cite{Pf} and classical references therein.
It should be mentioned that in this article $p$ is a positive prime integer.
We also employ notations of the author of the references
\cite{AGHF.SFWFPF, AGHF.wpt.ff, AGHF.wpt.n}.

For a finite  group $G$, the complex vector space
$\mathbb{C}^G=\{\mathbf{x}:G\to\mathbb{C}\}$ is a $|G|$-dimensional vector space with
complex vector entries indexed by elements in the finite group $G$.
The inner product of $\mathbf{x},\mathbf{y}\in\mathbb{C}^G$ is defined by
$\langle\mathbf{x},\mathbf{y}\rangle=\sum_{g\in G}\mathbf{x}(g)\overline{\mathbf{y}(g)}$,
and the induced norm is the $\|.\|_2$-norm of $\mathbf{x}$.
For $\mathbb{C}^{\mathbb{Z}_p}$,
where $\mathbb{Z}_p$ denotes the cyclic group of $p$ elements  $\{0,\ldots,p-1\}$, we write
$\mathbb{C}^p$. The notation $\|\mathbf{x}\|_0=|\{k\in\mathbb{Z}_p:\mathbf{x}(k)\not=0\}|$
counts non-zero entries in $\mathbf{x}\in\mathbb{C}^p$.
The translation operator $T_k:\mathbb{C}^p\to\mathbb{C}^p$ is
$T_k\mathbf{x}(s)=\mathbf{x}(s-k)$ for $\mathbf{x}\in\mathbb{C}^p$ and $k,s\in\mathbb{Z}_p$.
The modulation operator $M_\ell:\mathbb{C}^p\to\mathbb{C}^p$ is given by
$M_\ell\mathbf{x}(s)=e^{-2\pi i\ell s/p}\mathbf{x}(s)$ for $\mathbf{x}\in\mathbb{C}^p$ and
$\ell,s\in\mathbb{Z}_p$. The translation and modulation operators on the Hilbert space
$\mathbb{C}^p$ are unitary operators in the $\|.\|_2$-norm. For $\ell,k\in\mathbb{Z}_p$ we have
$(T_k)^*=(T_k)^{-1}=T_{p-k}$ and $(M_{\ell})^*=(M_{\ell})^{-1}=M_{p-\ell}$. The unitary discrete
Fourier Transform (DFT) of a 1D discrete signal $\mathbf{x}\in\mathbb{C}^p$ is defined by
$\widehat{\mathbf{x}}(\ell)=\frac{1}{\sqrt{p}}\sum_{k=0}^{p-1}\mathbf{x}(k)\overline{\mathbf{w}_{\ell}(k)},$
for all $\ell\in\mathbb{Z}_p$, where for all
$\ell,k\in\mathbb{Z}_p$ we have $\mathbf{w}_{\ell}(k) = e^{2\pi i\ell k/p}.$
Thus, DFT of $\mathbf{x}\in\mathbb{C}^p$ at $\ell\in \mathbb{Z}_p$ is
\begin{equation}\label{DFT}
\widehat{\mathbf{x}}(\ell)=\mathcal{F}_p(\mathbf{x})(\ell)
=\frac{1}{\sqrt{p}}\sum_{k=0}^{p-1}\mathbf{x}(k)\overline{\mathbf{w}_{\ell}(k)}
=\frac{1}{\sqrt{p}}\sum_{k=0}^{p-1}\mathbf{x}(k)e^{-2\pi i\ell k/p}.
\end{equation}
The DFT is a unitary transform in $\|.\|_2$-norm, i.e. for all
$\mathbf{x}\in \mathbb{C}^p$ satisfies the Parseval formula
$\|\mathcal{F}_{p}(\mathbf{x})\|_{2}=\|\mathbf{x}\|_{2}$.
The Polarization identity implies the Plancherel formula $\langle\mathbf{x},\mathbf{y}\rangle
=\langle\widehat{\mathbf{x}},\widehat{\mathbf{y}}\rangle$ for $\mathbf{x},\mathbf{y}\in\mathbb{C}^p$.
The unitary DFT (\ref{DFT}) satisfies
$\widehat{T_k\mathbf{x}}=M_k\widehat{\mathbf{x}}$, $\widehat{M_\ell\mathbf{x}}=T_{p-\ell}\widehat{\mathbf{x}}$,
for $\mathbf{x}\in\mathbb{C}^p$ and $k,\ell\in\mathbb{Z}_p$.
Also the inverse discrete Fourier formula (IDFT) for a 1D discrete signal
$\mathbf{x}\in\mathbb{C}^p$ is given by
\begin{equation*}
\mathbf{x}(k)=\frac{1}{\sqrt{p}}\sum_{\ell=0}^{p-1}\widehat{\mathbf{x}}(\ell)\mathbf{w}_{\ell}(k)=
\frac{1}{\sqrt{p}}\sum_{\ell=0}^{p-1}\widehat{\mathbf{x}}(\ell)e^{2\pi i\ell k/p},\ \ 0\le k\le p-1.
\end{equation*}

A finite sequence $\mathfrak{A}=\{\mathbf{y}_j:0\le j\le M-1\}\subset\mathbb{C}^p$
is called a frame (or finite frame) for $\mathbb{C}^p$,
if there exists a positive constant $0<A<\infty$ such that  \cite{ff}
\begin{equation*}\label{f.c}
A\|\mathbf{x}\|_{2}^2\le \sum_{j=0}^{M-1}|\langle
\mathbf{x},\mathbf{y}_j\rangle|^2,\  \  \ (\mathbf{x}\in\mathbb{C}^{p}).
\end{equation*}
 The synthesis operator
$F:\mathbb{C}^M\to\mathbb{C}^p$ is defined by
 $F\{c_j\}_{j=0}^{M-1}=\sum_{j=0}^{M-1}c_j\mathbf{y}_j$ for all
$\{c_j\}_{j=0}^{M-1}\in\mathbb{C}^M$.
The adjoint operator $F^*:\mathbb{C}^p\to\mathbb{C}^M$ is defined by
$F^*\mathbf{x}=\{\langle\mathbf{x},\mathbf{y}_j\rangle\}_{j=0}^{M-1}$ for all $\mathbf{x}\in\mathbb{C}^p$.
If $\mathfrak{A}=\{\mathbf{y}_j:0\le j\le M-1\}$ is a frame for $\mathbb{C}^p$, by composing $F$ and $F^*$, we get
the positive and invertible frame operator $S:\mathbb{C}^p\to\mathbb{C}^p$ given by
\begin{equation*}\label{S}
\mathbf{x}\mapsto S\mathbf{x}=FF^*\mathbf{x}
=\sum_{j=0}^{M-1}\langle\mathbf{x},\mathbf{y}_j\rangle\mathbf{y}_j,\  \  \  \ (\mathbf{x}\in\mathbb{C}^{p}),
\end{equation*}
and
the set $\mathfrak{A}$ spans the complex Hilbert space $\mathbb{C}^p$ which implies $M\ge p$, where $M=|\mathfrak{A}|$.
Each finite spanning set in $\mathbb{C}^p$ is a finite frame for $\mathbb{C}^p$. The ratio
between $M$ and $p$ is called the redundancy of the finite frame $\mathfrak{A}$ (i.e. ${\rm red}_{\mathfrak{A}}=M/p$).
If $\mathfrak{A}$ is a finite frame for $\mathbb{C}^p$, each
 $\mathbf{x}\in\mathbb{C}^p$
satisfies
\begin{equation*}\label{f.r}
\mathbf{x}=\sum_{j=0}^{M-1}\langle\mathbf{x},S^{-1}\mathbf{y}_j\rangle\mathbf{y}_j
=\sum_{j=0}^{M-1}\langle\mathbf{x},\mathbf{y}_j\rangle S^{-1}\mathbf{y}_j.
\end{equation*}

\section{\bf Construction of Wavelet Frames over Prime Fields}

In this section, we briefly state structure and basic properties of
cyclic dilation operators, see \cite{f.g.h.t, AGHF.SFWFPF, AGHF, j}.
Then we present an overview for the notion and structure of
finite wavelet groups over prime fields \cite{ff0, ram.val}.

\subsection{\bf Structure of Finite Wavelet Group over Prime Fields}

The set
\begin{equation*}
\mathbb{U}_p:=\mathbb{Z}_p-\{0\}=\{1,...,p-1\},
\end{equation*}
is a finite multiplicative Abelian group of order $p-1$
with respect to the multiplication module $p$ with the identity element ${1}$.
The multiplicative inverse for $m\in\mathbb{U}_p$
is $m_p$ which satisfies $m_pm+np=1$ for some $n\in\mathbb{Z}$,
see \cite{Hardy, Rie}.

For $m\in\mathbb{U}_p$,  the cyclic dilation operator is defined by
$D_m:\mathbb{C}^p\to\mathbb{C}^p$ by
\begin{align*}
D _m\mathbf{x}(k):=\mathbf{x}(m_pk)
\end{align*}
for $\mathbf{x}\in\mathbb{C}^p$ and $k\in\mathbb{Z}_p$,
where $m_p$ is the multiplicative inverse of $m$ in $\mathbb{U}_p$.

In the following propositions, we state some basic properties of cyclic dilations.

\begin{proposition}\label{prop30}
{\it Let $p$ be a positive prime integer. Then
\begin{enumerate}[\normalfont (i)]
\item For $(m,k)\in\mathbb{U}_p\times\mathbb{Z}_p$, we have $D_mT_k=T_{mk}D_m$.
\item For $m,m'\in\mathbb{U}_p$, we have $D_{mm'}=D_mD_{m'}$.
\item For $(m,k),(m',k')\in\mathbb{U}_p\times\mathbb{Z}_p$, we have
$T_{k+mk'}D_{mm'}=T_kD_mT_{k'}D_{m'}$.
\item For $(m,\ell)\in \mathbb{U}_p\times\mathbb{Z}_p$, we have $D_mM_\ell=M_{m_p\ell}D_m$.
\end{enumerate}
}\end{proposition}
The next result states some properties of cyclic dilations.

\begin{proposition}\label{prop31}
{\it Let $p$ be a positive prime integer and $m\in\mathbb{U}_p$.
Then
\begin{enumerate}[\normalfont (i)]
\item The dilation operator $D_m:\mathbb{C}^p\to\mathbb{C}^p$ is a $*$-homomorphism.
\item The dilation operator $D_m:\mathbb{C}^p\to\mathbb{C}^p$ is
unitary in $\|.\|_2$-norm and satisfies
\[
(D_m)^*=(D_m)^{-1}=D_{m_p}.
\]
\item For $\mathbf{x}\in\mathbb{C}^p$, we have
$\widehat{D_m\mathbf{x}}=D_{m_p}\widehat{\mathbf{x}}$.
\end{enumerate}
}\end{proposition}
The underlying set
$$\mathbb{U}_p\times\mathbb{Z}_p=\left\{(m,k):m\in\mathbb{U}_p, k\in\mathbb{Z}_p\right\},$$
equipped with the following group operations
\begin{equation*}
(m,k)\ltimes (m',k'):=(mm',k+mk'),
\end{equation*}
\begin{equation*}
(m,k)^{-1}:=(m_p,m_p.(p-k)),
\end{equation*}
is a finite non-Abelian group of order $p(p-1)$ denoted by $\mathbb{W}_p$ and it is called as
finite affine group on $p$ integers or
 the finite wavelet group associated to the integer $p$ or simply as $p$-wavelet group.

Next proposition states basic properties of the finite wavelet group $\mathbb{W}_p$.

\begin{proposition}
Let $p>2$ be a positive prime integer. Then $\mathbb{W}_p$ is a non-Abelian group of order $p(p-1)$ which contains
 a copy of $\mathbb{Z}_p$ as a normal Abelian subgroup and a copy of $\mathbb{U}_p$
 as a non-normal Abelian subgroup.
\end{proposition}

\subsection{\bf Wavelet Frames over Prime Fields}

A finite wavelet system for the complex Hilbert space $\mathbb{C}^p$ is
a family or system of the form
\begin{equation*}
\mathcal{W}({\mathbf{y}},\Delta)
:=\left\{\sigma(m,k)\mathbf{y}=T_kD_m\mathbf{y}:(m,k)\in\Delta\subseteq\mathbb{W}_p\right\},
\end{equation*}
for some window signal $\mathbf{y}\in\mathbb{C}^p$ and a subset $\Delta$ of $\mathbb{W}_p$.

If $\Delta=\mathbb{W}_p$, we put $\mathcal{W}(\mathbf{y}):=\mathcal{W}(\mathbf{y},\mathbb{W}_p)$
and it is called a full finite wavelet system over $\mathbb{Z}_p$.
A finite wavelet system which spans $\mathbb{C}^p$ is a frame
and is referred to as a finite wavelet frame over the prime field $\mathbb{Z}_p$.

If $\mathbf{y}\in\mathbb{C}^p$ is a window signal then for $\mathbf{x}\in\mathbb{C}^p$, we have
\begin{equation*}\label{w.r.i}
\langle\mathbf{x},\sigma(m,k){\mathbf{y}}\rangle
=\langle\mathbf{x},T_kD_m{\mathbf{y}}\rangle
=\langle T_{p-k}\mathbf{x},D_m{\mathbf{y}}\rangle,\ \ \  \ ((m,k)\in\mathbb{W}_p).
\end{equation*}

The following proposition states a formulation for wavelet coefficients via Fourier transform.

\begin{proposition}\label{w.f.rep}
{\it Let $\mathbf{x},\mathbf{y}\in\mathbb{C}^p$ and $(m,k)\in\mathbb{W}_{p}$. Then,
$$\langle\mathbf{x},\sigma(m,k){\mathbf{y}}\rangle
=\sqrt{p}\mathcal{F}_p(\widehat{\mathbf{x}}.\overline{\widehat{D_m{\mathbf{y}}}})(p-k).$$
}\end{proposition}
\begin{proof}
See Proposition 4.1 of \cite{AGHF.SFWFPF}.
\end{proof}

In \cite{AGHF.SFWFPF} using an analytic approach the author has presented a
concrete formulation for the $\|.\|_2$-norm of wavelet coefficients,
the formula of which is just stated hereby.

\begin{theorem}\label{f.f.s}
Let $p$ be a positive prime integer, $M$  a divisor of $p-1$ and
$\mathbb{M}$ be a multiplicative subgroup of $\mathbb{U}_p$ of order $M$.
Let $\mathbf{y}\in\mathbb{C}^p$ be a window signal and $\mathbf{x}\in\mathbb{C}^p$. Then,
\begin{align*}
&\sum_{m\in\mathbb{M}}\sum_{k\in\mathbb{Z}_p}|\langle\mathbf{x},\sigma(m,k)\mathbf{y}\rangle|^2
\\&=p\left(M|\widehat{{\mathbf{y}}}(0)|^2|\widehat{\mathbf{x}}(0)|^2
+\left(\sum_{m\in\mathbb{M}}|\widehat{{\mathbf{y}}}(m)|^2\right)\left(\sum_{\ell\in\mathbb{M}}|\widehat{\mathbf{x}}(\ell)|^2\right)
+\sum_{\ell\in\mathbb{U}_p-\mathbb{M}}\gamma_\ell(\mathbf{y},\mathbb{M})|\widehat{\mathbf{x}}(\ell)|^2\right),
\end{align*}
where
\[
\gamma_\ell(\mathbf{y},\mathbb{M}):=\sum_{m\in\mathbb{M}}|\widehat{\mathbf{y}}(m\ell)|^2,\ \ \ \ \ (\ell\in\mathbb{U}_p-\mathbb{M}).
\]
\end{theorem}

\begin{proof}
 See Theorem 4.2 of \cite{AGHF.SFWFPF}.
\end{proof}

\begin{remark}
The formulation presented in Theorem \ref{f.f.s} is an analytic
formulation of wavelet coefficients associated to the subgroup $\mathbb{M}$.
In details, that formulation originated from an analytic
approach which was based on direct computations of cyclic dilations in the subgroup $\mathbb{M}$.
\end{remark}

In the following theorem, we present a constructive formulation for the $\|.\|_2$-norm of wavelet coefficients.
At the first, we need to prove some results.

\begin{proposition}\label{part}
Let $p$ be a positive prime integer, $M$  a divisor of $p-1$
and $\mathbb{M}$  a multiplicative subgroup of $\mathbb{U}_p$ of order $M$.
Let $\epsilon$ be a generator of the cyclic group $\mathbb{U}_p$ and
$a:=\frac{p-1}{M}$  the index of $\mathbb{M}$ in $\mathbb{U}_p$. Then
\begin{enumerate}[\normalfont (i)]
\item For $0\le r,s\le a-1$, we have $\epsilon^r\mathbb{M}=\epsilon^s\mathbb{M}$ iff $r=s$.
\item $\mathbb{U}_p/\mathbb{M}=\{\epsilon^t\mathbb{M}:0\le t\le a-1\}$.
\end{enumerate}
\end{proposition}
\begin{proof}
(i)
First, suppose $\epsilon^r \mathbb{M}= \epsilon^ s \mathbb{M}$
for some $0\leq r,s\leq a-1.$ We will show
$r=s.$
For see this, assume to the contrary that $r \neq s$.
Without loss of generality, suppose $s > r$. Thus, there exists
\begin{align}\label{dis2}
1 \leq r' \leq a-1 < a,
\end{align}
such that
$s= r+r'$, and we get ${\epsilon}^{r} \mathbb{M}={\epsilon}^{r+r'} \mathbb{M}={\epsilon}^{r} {\epsilon}^{r'}\mathbb{M}$.
Since ${\epsilon}^{r}$ is an invertible element of $\mathbb{U}_{p}$,
so $\mathbb{M}={\epsilon}^{r'}\mathbb{M}$
and by this ${\epsilon}^{r'}\in \mathbb{M}.$
Also,
$\epsilon^{p-1}\stackrel{p}{\equiv}({\epsilon}^{r'})^{\frac{p-1}{r'}}\stackrel{p}{\equiv} 1 .$
Hence
the order of ${\epsilon}^{r'}$ in $\mathbb{U}_{p}$ can be at most
$\frac{p-1}{r'}$. Since  ${\epsilon}$ is a generator of
$\mathbb{U}_{p}$, the order of ${\epsilon}^{r'}$ is $\frac{p-1}{r'}$. In addition, by
(\ref{dis2}) we get $M < {\frac{p-1}{r'}}\leq p-1$, therefore
 ${\epsilon}^{r'}$ can not be in $\mathbb{M}$ which is a contradiction.
The inverse implication is straightforward.

(ii) According to (i), $\{\epsilon^t \mathbb{M} : 0\leq t \leq a-1\}$
is a set of disjoint cosets of $\mathbb{M}$ in $\mathbb{U}_p$ and the cardinality
of this set is equal to the index number of $\mathbb{M}$ in $\mathbb{U}_p$,
so the equality in (ii) holds. In particular, this set of cosets  creates a partition to $\mathbb{U}_p$.
\end{proof}
Next we present a constructive formula for $\|.\|_2$ of wavelet coefficients.
\begin{theorem}\label{f.f.s.coset}
Let $p$ be a positive prime integer, $M$  a divisor of $p-1$
and $\mathbb{M}$  a multiplicative subgroup of $\mathbb{U}_p$ of order $M$.
Let $\epsilon$ be a generator of the cyclic group $\mathbb{U}_p$ and
$a:=\frac{p-1}{M}$  the index of $\mathbb{M}$ in $\mathbb{U}_p$.
Let $\mathbf{y}\in\mathbb{C}^p$ be a window signal and
$\mathbf{x}\in\mathbb{C}^p$.
Then,
\begin{align*}\label{expan.coset}
&\sum_{m\in\mathbb{M}}\sum_{k=0}^{p-1}|\langle\mathbf{x},T_kD_m\mathbf{y}\rangle|^2
\\&\hspace{1cm}=p\bigg(M{|\widehat{\mathbf{x}}(0)|}^{2} {|\widehat{\mathbf{y}}(0)|}^{2}
+ \sum_{t=0}^{a-1}\bigg(\sum_{\ell\in H_t}{|\widehat{\mathbf{x}}(\ell)|}^{2}\bigg)
\bigg(\sum_{w\in H_t}{|\widehat{\mathbf{y}}(w)|}^{2}\bigg)\bigg),
\end{align*}
where $H_t:=\epsilon^t \mathbb{M}$ for all $0\le t\le a-1$.
\end{theorem}
\begin{proof}
Let $\mathbf{y}\in\mathbb{C}^p$ be a window function, $\mathbf{x}\in\mathbb{C}^p$ and
$m\in \mathbb{M}$. Then using Proposition \ref{w.f.rep} and Plancherel formula we have
\begin{align*}
\sum_{k=0}^{p-1}|\langle\mathbf{x},T_kD_m\mathbf{y}\rangle|^2
=p\sum_{\ell=0}^{p-1}|\widehat{\mathbf{x}}(\ell)|^2 . |\overline{\widehat{ {\mathbf{y}}}}(m \ell)|^2.
\end{align*}
Hence we achieve
\begin{align*}
\sum_{m\in \mathbb{M}}\sum_{k=0}^{p-1}|\langle\mathbf{x},T_kD_m\mathbf{y}\rangle|^2
&=\sum_{m\in \mathbb{M}}p \sum_{\ell=0}^{p-1}|\widehat{\mathbf{x}}(\ell)|^2.|\overline{\widehat{ {\mathbf{y}}}}(m \ell)|^2
\\&=p \sum_{\ell=0}^{p-1}\sum_{m\in \mathbb{M}}|\widehat{\mathbf{x}}(\ell)|^2.|\overline{\widehat{ {\mathbf{y}}}}(m \ell)|^2.
\end{align*}
Then we have
\begin{equation}\label{expan}
\sum_{\ell=0}^{p-1}|\widehat{\mathbf{x}}(\ell)|^2\bigg(\sum_{m\in \mathbb{M}}|\widehat{\mathbf{y}}(m\ell)|^2\bigg)
=|\widehat{\mathbf{x}}(0)|^2\bigg(\sum_{m\in \mathbb{M}}|\widehat{\mathbf{y}}(0)|^2\bigg)
+\sum_{\ell=1}^{p-1}|
\widehat{\mathbf{x}}(\ell)|^2\bigg(\sum_{m\in \mathbb{M}}|\widehat{\mathbf{y}}(m\ell)|^2\bigg).
\end{equation}
Now by (\ref{expan}) and Proposition \ref{part}, we get
\begin{align*}
\sum_{m\in \mathbb{M}}\sum_{\ell=0}^{p-1}|\widehat{\mathbf{x}}(\ell)|^2|\widehat{\mathbf{y}}(m\ell)|^2
&=M|\widehat{\mathbf{x}}(0)|^2|\widehat{\mathbf{y}}(0)|^2\\&+\sum_{t=0}^{a-1}\sum_{\ell\in H_t}|
\widehat{\mathbf{x}}(\ell)|^2\bigg(\sum_{m\in \mathbb{M}}|\widehat{\mathbf{y}}(m\ell)|^2\bigg).
\end{align*}
Replacing $m\ell$ with $w$ we get,
$w\stackrel{p}{\equiv} m\ell \in \mathbb{M}.H_t=\mathbb{M}.{\epsilon}^t\mathbb{M}= {\epsilon}^t\mathbb{M}=H_t.$
For each $\ell\in H_{t}$, the mapping $m\rightarrow m\ell$ defines a bijection of $M$ onto $H_t$, and thus we deduce that
\begin{align*}
\sum_{m\in \mathbb{M}}\sum_{\ell=0}^{p-1}|\widehat{\mathbf{x}}(\ell)|^2|\widehat{\mathbf{y}}(m\ell)|^2
&=M|\widehat{\mathbf{x}}(0)|^2|\widehat{\mathbf{y}}(0)|^2\\&+\sum_{t=0}^{a-1}\bigg(\sum_{\ell\in H_t}|
\widehat{\mathbf{x}}(\ell)|^2\bigg)\bigg(\sum_{w\in H_t}|\widehat{\mathbf{y}}(w)|^2\bigg).
\end{align*}
Therefore
\begin{align*}
&\sum_{m\in\mathbb{M}}\sum_{k=0}^{p-1}|\langle\mathbf{x},T_kD_m\mathbf{y}\rangle|^2
\\&=p\bigg(M{|\widehat{\mathbf{x}}(0)|}^{2} {|\widehat{\mathbf{y}}(0)|}^{2}
+ \sum_{t=0}^{a-1}\bigg(\sum_{\ell\in H_t}{|\widehat{\mathbf{x}}(\ell)|}^{2}\bigg)
\bigg(\sum_{w\in H_t}{|\widehat{\mathbf{y}}(w)|}^{2}\bigg)\bigg).
\end{align*}
\end{proof}
The following result is an immediate consequence of Theorem \ref{f.f.s.coset}.
\begin{corollary}\label{f.s}
{\it Let $p$ be a positive prime integer. Let $\mathbf{y}\in\mathbb{C}^p$ be a window
signal and $\mathbf{x}\in\mathbb{C}^p$.
Then
\begin{align*}
&\sum_{m\in\mathbb{U}_p}\sum_{k=0}^{p-1}|\langle\mathbf{x},T_kD_m\mathbf{y}\rangle|^2
\\&=p\bigg((p-1){|\widehat{\mathbf{x}}(0)|}^{2} {|\widehat{\mathbf{y}}(0)|}^{2}
+ \bigg(\sum_{\ell\in \mathbb{U}_p}{|\widehat{\mathbf{x}}(\ell)|}^{2}\bigg)
\bigg(\sum_{w\in \mathbb{U}_p}{|\widehat{\mathbf{y}}(w)|}^{2}\bigg)\bigg).
\end{align*}
}
\end{corollary}
Applying Theorem \ref{f.f.s.coset}, we can present the following constructive characterization
of a given window signal $\mathbf{y}\in\mathbb{C}^p$ and a subgroup of
the finite wavelet group $\mathbb{W}_p$ to guarantee
that generated wavelet system is a frame for $\mathbb{C}^p$.
\begin{theorem}\label{m.f.result}
Let $p$ be a positive prime integer, $\epsilon$  a generator of
$\mathbb{U}_p$,
 $M$  a divisor of $p-1$,
$\mathbb{M}$  a multiplicative subgroup of $\mathbb{U}_p$ of order $M$ and
$a:=\frac{p-1}{M}$ the index of $\mathbb{M}$ in $\mathbb{U}_p$.
Let $\Delta_\mathbb{M}:=\mathbb{M}\times\mathbb{Z}_p$ and $\mathbf{y}\in\mathbb{C}^p$ be a non-zero window signal.
The finite wavelet system $\mathcal{W}(\mathbf{y},\Delta_\mathbb{M})$ is a frame for
$\mathbb{C}^p$ if and only if the following conditions hold
\begin{enumerate}[\normalfont (i)]
\item $\widehat{\mathbf{y}}(0)\neq 0$
\item For each $t\in \{0,...,a-1\}$, there exists $m_t\in \mathbb{M}$ such that  $\widehat{\mathbf{y}}(\epsilon^t m_t )\neq 0$.
\end{enumerate}
\end{theorem}
\begin{proof}
\rm
Let $\mathbf{y}$ be a non-zero window signal which satisfies conditions (i), (ii). Then by definition of
$H_t$ we get
\begin{align*}
\vartheta:= \min_{t\in \{0,...,a-1\}}\left(\sum_{m\in \mathbb{M}} |\widehat{\mathbf{y}}(m\epsilon^t)|^2\right)
=\min_{t\in \{0,...,a-1\}}\left(\sum_{w\in H_t} |\widehat{\mathbf{y}}(w)|^2\right)\neq 0.
\end{align*}
Now, let $0< A < \infty$ be given by
\begin{equation*}\label{low}
A:=\min \bigg\{M \Big|\sum_{k=0}^{p-1} \mathbf{y}(k)\Big|^2, p \vartheta \bigg\}.
\end{equation*}
Using Theorem \ref{f.f.s.coset} for $\mathbf{x}\in \mathbb{C}^p$, we have
\begin{align*}
\sum_{m\in\mathbb{M}}\sum_{k=0}^{p-1}|\langle\mathbf{x},T_kD_m\mathbf{y}\rangle|^2
&=p M|\widehat{\mathbf{x}}(0)|^2|\widehat{\mathbf{y}}(0)|^2+ p\sum_{t=0}^{a-1}\bigg(\sum_{\ell\in H_t}|
\widehat{\mathbf{x}}(\ell)|^2\bigg)\bigg(\sum_{w\in H_t}|\widehat{\mathbf{y}}(w)|^2\bigg)
\\& \geq p M|\widehat{\mathbf{x}}(0)|^2|\widehat{\mathbf{y}}(0)|^2+ p \vartheta\sum_{t=0}^{a-1}\bigg(\sum_{\ell\in H_t}|
\widehat{\mathbf{x}}(\ell)|^2\bigg)
\\& \geq A |\widehat{\mathbf{x}}(0)|^2 + A \sum_{t=0}^{a-1}\bigg(\sum_{\ell\in H_t}|
\widehat{\mathbf{x}}(\ell)|^2\bigg)
\\&=A |\widehat{\mathbf{x}}(0)|^2 + A \sum_{\ell=1}^{p-1}
|\widehat{\mathbf{x}}(\ell)|^2\\& = A \|\widehat{\mathbf{x}}\|_{2}^2= A \|\mathbf{x}\|_{2}^2.
\end{align*}
which implies that the finite
wavelet system $\mathcal{W}(\mathbf{y},\Delta_\mathbb{M})$ is a frame for $\mathbb{C}^p$.

For the inverse implication, consider $\mathbf{y}\in \mathbb{C}^p$ be a non-zero window signal such that
the finite wavelet system $\mathcal{W}(\mathbf{y},\Delta_\mathbb{M})$ is a frame for
$\mathbb{C}^p$. Thus, there exists $ A_1>0$ such that
\begin{align*}
&\sum_{m\in\mathbb{M}}\sum_{k=0}^{p-1}|\langle\mathbf{x},T_kD_m\mathbf{y}\rangle|^2 \geq A_1 \|\mathbf{x}\|_{2}^2,\hspace{1cm}
(\mathbf{x}\in \mathbb{C}^p).
\end{align*}
Then by Theorem \ref{f.f.s.coset} we have
\begin{align}\label{A_2}
M|\widehat{\mathbf{y}}(0)|^2|\widehat{\mathbf{x}}(0)|^2+\sum_{t=0}^{a-1}\bigg(\sum_{\ell \in H_t}|\widehat{\mathbf{x}}(\ell)|^2\bigg)\bigg(\sum_{w \in H_t}|\widehat{\mathbf{y}}(w)|^2\bigg)
\geq A_2 \|\mathbf{x}\|_{2}^2
\end{align}
for all $\mathbf{x}\in \mathbb{C}^p$, where $A_2=\frac{A_1}{p}$.
Now let $\mathbf{x}'\in \mathbb{C}^p$ with
$
\widehat{\mathbf{x}}'(0)\neq 0
$
and
$\widehat{\mathbf{x}}'(\ell)=0$, for all $\ell\in\{1,...,p-1\}$.
Thus, by  (\ref{A_2}) we get $\widehat{\mathbf{y}}(0)\neq 0$.
Next, consider $\mathbf{x}_t\in \mathbb{C}^p$ be a non-zero vector in which for a fixed but arbitrary $t\in\{0,..., a-1\}$,
\begin{equation}\label{inv}
{\widehat{\mathbf{x}}}_t(\ell)=0,\forall\ell\not\in H_t.
\end{equation}
Hence (\ref{inv}) assures that
$\sum_{w\in H_t}|\widehat{\mathbf{y}}(w)|^2$ should be non-zero. Therefore
$\widehat{y}(w)\neq 0$
for some $w \in H_t$, which indicates that there exists $m_t\in \mathbb{M}$
such that $\widehat{\mathbf{y}}(m_t\epsilon^t)\neq 0.$
\end{proof}
The following result shows that for a large class of non-zero window signals
the finite wavelet system $\mathcal{W}(\mathbf{y})$ is a frame for
$\mathbb{C}^p$ with redundancy $p-1$.

\begin{corollary}\label{th11}
{\it Let $p$ be a positive prime integer and $\mathbf{y}\in\mathbb{C}^p$ be a non-zero window signal.
The finite wavelet system $\mathcal{W}(\mathbf{y})$ constitutes a frame for
$\mathbb{C}^p$ with the redundancy $p-1$ if and only if $\widehat{\mathbf{y}}(0)\not =0$
and $\|\widehat{\mathbf{y}}\|_0\ge 2$.}
\end{corollary}

The following corollary presents a characterization for finite wavelet systems
over prime fields using matrix analysis language.

\begin{corollary}\label{matrix.int}
Let $p$ be a positive prime integer, $\epsilon$  a generator of
$\mathbb{U}_p$,
 $M$  a divisor of $p-1$,
$\mathbb{M}$  a multiplicative subgroup of $\mathbb{U}_p$ of order $M$ and
$a:=\frac{p-1}{M}$ be the index of $\mathbb{M}$ in $\mathbb{U}_p$.
Let $\Delta_\mathbb{M}:=\mathbb{M}\times\mathbb{Z}_p$ and $\mathbf{y}\in\mathbb{C}^p$ be a non-zero window signal.
The finite wavelet system $\mathcal{W}(\mathbf{y},\Delta_\mathbb{M})$ is a frame for
$\mathbb{C}^p$ if and only if  $\widehat{\mathbf{y}}(0)\neq 0$ and the matrix ${\mathbf{Y}}({\mathbb{M},\mathbf{y}})$ of size $a\times M$ given by
$$
{\mathbf{Y}}({\mathbb{M},\mathbf{y}}):=\begin{bmatrix}
 \widehat{\mathbf {y}}(1)& \widehat{\mathbf {y}}(\epsilon^{a})& \cdots & \widehat{\mathbf {y}}(\epsilon^{(M-1)a}) \\
 \widehat{\mathbf {y}}(\epsilon^{1})& \widehat{\mathbf {y}}(\epsilon^{a+1})& \cdots & \widehat{\mathbf {y}}(\epsilon^{(M-1)a+1}) \\
 \vdots & \vdots & \ddots & \vdots\\
  \widehat{\mathbf {y}}(\epsilon^{a-1})& \widehat{\mathbf {y}}(\epsilon^{2a-1})&  \cdots&\widehat{\mathbf {y}}(\epsilon^{Ma-1})
 \end{bmatrix}_{a\times M}
 $$
is a matrix such that each row has at least a non-zero entry.
\end{corollary}
\begin{proof}
Since
$\epsilon^a$ is a generator of $\mathbb{M}$
and $|\mathbb{M}|=M$,
 so
 \begin{equation}\label{m}
 \mathbb{M}=\{(\epsilon^a)^0,(\epsilon^a)^1,...,(\epsilon^a)^{M-1}\}=\{1,\epsilon^a,...,\epsilon^{(M-1)a}\}.
 \end{equation}
 By (\ref{m}) and definition of $H_t$,
 we have
 \begin{equation*}
 H_t=\{\epsilon^t,\epsilon^a.\epsilon^t,..., \epsilon^{(M-1)a}. \epsilon^t\}=\{\epsilon^t,\epsilon^{a+t},...,\epsilon^{(M-1)a+t}\}
 \end{equation*}
 for all $t\in\{0,...,a-1\}$. The entries of $t$-th row of ${\mathbf{Y}}({\mathbb{M},\mathbf{y}})$ are entries
 of $\widehat{\mathbf{y}}$ on $H_t$. Hence by Theorem \ref{m.f.result} the result holds.
\end{proof}

We can also deduce the following tight frame condition for finite wavelet systems
generated by subgroups of $\mathbb{W}_p$.

\begin{proposition}\label{tight}
{\it Let $p$ be a positive prime integer, $M$ a divisor of $p-1$ and
$\mathbb{M}$  a multiplicative subgroup of $\mathbb{U}_p$ of order $M$.
Let $\Delta_\mathbb{M}:=\mathbb{M}\times\mathbb{Z}_p$ and $\mathbf{y}\in\mathbb{C}^p$ be a non-zero window signal.
The finite wavelet system $\mathcal{W}(\mathbf{y},\Delta_\mathbb{M})$ is a tight frame for
$\mathbb{C}^p$ if and only if $\widehat{\mathbf{y}}(0)\not=0$ and
\begin{equation*}\label{t.f.c}
M \left|\sum_{k=0}^{p-1} \mathbf{y}(k)\right|^2
=p\left(\sum_{w\in H_t} |\widehat{\mathbf{y}}(w)|^2\right)
\end{equation*}
for all $t\in \{0,...,a-1\}$.
 In this case
\[
 \alpha_\mathbf{y}:=pM|\widehat{\mathbf{y}}(0)|^2=p\sum_{m\in\mathbb{M}}|\widehat{\mathbf{y}}(m)|^2
\]
is the frame bound.
}\end{proposition}
\begin{proof}
It can be proven by a similar argument used in Theorem \ref{m.f.result}.
\end{proof}
\noindent\textbf{Acknowledgement} 

\vspace{.5cm}
Some of the results are obtained during  the second author's appointment from the NuHAG group at the University of Vienna. We would like to thank Prof. Hans. G. Feichtinger for his valuable comments and the group for their hospitality.

\hspace{1in}

\end{document}